\documentclass[smallextended,envcountsect]{svjour3}
\smartqed
\usepackage{graphicx}
\journalname{JOTA}

\usepackage{amsmath, amsfonts,amssymb} 

\usepackage{mathtools}   
\usepackage{algorithm}
\usepackage{algorithmic}

\usepackage{xspace}
\usepackage{scalefnt}
\newcommand{\algn}[1]{{\sf
\scalefont{0.96}{#1}}\xspace}

\usepackage{amsmath, amsfonts,amssymb}
\usepackage{cite}

\newcommand{\sqnorm}[1]{\left\| #1 \right\|^2}
\newcommand{\Exp}[1]{\mathbb{E}\!\left[ #1 \right]}

\usepackage{color}
\definecolor{myblue}{rgb}{0,0.0,0.8}

\begin{document}

\title{Convergence Analysis of the ProbAbilistic Gradient Estimator Algorithm for Weakly Convex Finite-Sum Optimization}


\author{Laurent Condat and  Peter Richt\'{a}rik}

\institute{Laurent Condat and Peter Richt\'{a}rik \at
             King Abdullah University of Science and Technology (KAUST) \\
Thuwal, Kingdom of Saudi Arabia\\
              first.last@kaust.edu.sa
}

\date{Accepted: July 5th, 2026}

\maketitle

\begin{abstract}
The ProbAbilistic Gradient Estimator algorithm (PAGE), a stochastic algorithm introduced by Li et al.\ in 2021, was designed to find stationary points for the average of smooth nonconvex functions. In this work, we study PAGE within the broad framework of $\tau$-weakly convex functions, 
providing a continuous interpolation between the general nonconvex $L$-smooth regime  ($\tau=L$)  and the convex regime ($\tau=0$).  We establish new convergence rates for PAGE, showing that its complexity improves as $\tau$ decreases.
\end{abstract}


\section{Introduction}

Finite-sum optimization, which consists of minimizing the sum of a large number of functions, arises in a wide range of applications. A central example is 
empirical risk minimization, a fundamental 
model in machine learning and data science \cite{sra11,shai_book}.
Formally, let $\mathcal{X}$ be a finite-dimensional real Hilbert space. Given a number  $n\geq 1$ of smooth functions $f_i:\mathcal{X}\rightarrow \mathbb{R}$, $i\in[n]\coloneqq \{1,\ldots,n\}$, 
 the finite-sum optimization problem is
\begin{equation}
\min_{x\in\mathcal{X}} \,f(x)\coloneqq \frac{1}{n}\sum_{i=1}^n f_i(x).\label{eqp1}
\end{equation}
When the problem is nonconvex, we cannot expect to find a minimizer of $f$,
and the goal is instead to find a stationary point, where $\nabla f$ is zero. 
 A straightforward approach is Gradient Descent (GD), but it is often impractical: computing the full gradient $\nabla f(x)$ requires evaluating all $n$ individual gradients  $\nabla f_i(x)$, which becomes prohibitively expensive when $n$ is large.
 Stochastic algorithms overcome this challenge by evaluating only a small number of randomly chosen component gradients 
per iteration, offering a far more scalable alternative.

We make the following assumptions on the functions.
\begin{itemize}
\item There exists $L>0$ such that $f_i$ is $L$-smooth for every $i\in[n]$; that is, $f_i$ is differentiable on $\mathcal{X}$ and for every $(x,y)\in\mathcal{X}^2$, $\|\nabla f_i(x)-\nabla f_i(y)\|\leq L \|x-y\|$.
\item There exists $\tau\in [0,L]$ such that $f_i$ is $\tau$-weakly convex for every $i\in[n]$; that is, $f_i+\frac{\tau}{2} \|\cdot\|^2$ is convex \cite{nur73}.
\item $f$ is lower bounded: $f^\star \coloneqq \inf f > -\infty$. 
\end{itemize}
Note that every $L$-smooth function is $L$-weakly convex (Lemma~\ref{lemm1}). Furthermore, a twice continuously differentiable function is $\tau$-weakly convex if the minimal eigenvalue of its Hessian is uniformly bounded below by $-\tau$. Thus, the weak convexity
assumption is automatically satisfied, but our  analysis will reveal that convergence improves
as $\tau$ decreases from $L$ (general nonconvex regime) to 0 (convex regime). Weakly convex functions appear in a broad range of applications  \cite{dav19,sun19}.

In addition, we establish linear convergence when 
$f$ satisfies the 
Polyak--{\L}ojasiewicz (P{\L}) condition with constant $\mu>0$, or in short is $\mu$-P{\L} \cite{kar16}:
 \begin{equation*}
\|\nabla f(x)\|^2\geq 2\mu\big(f(x)-f^\star\big),\ \mbox{for every }x\in\mathcal{X}.
\end{equation*}
In this case, $\mu\leq L$ and we define the condition number $\kappa\coloneqq \frac{L}{\mu}\geq 1$.  
Note that if $f$ is $\mu$-strongly convex, i.e., $f-\frac{\mu}{2}\|\cdot\|^2$ is convex, then it is $\mu$-P{\L} \cite{kar16}. 
But the reverse need not hold. For example, the function 
$(x_1,x_2)\in\mathbb{R}^2\mapsto \frac{1}{2}x_1^2 $ is 1-P{\L} but not strongly convex, illustrating that the P{\L} condition is a strictly weaker requirement.
Crucially, the P{\L} condition often holds in nonconvex problems of practical interest, including low-rank matrix recovery \cite{bho16}, phase retrieval \cite{sun16}, and deep learning \cite{kaw16}. In these cases, it guarantees the absence of spurious local minima and underpins fast convergence.

Our assumptions are mild. In particular, we do not assume that a minimizer of $f$ exists. For example, the logistic loss  $x\in\mathbb{R}\mapsto \log(1+e^x)$,  widely used in machine learning, is convex, $\frac{1}{4}$-smooth, and bounded below by $f^\star=0$, yet it has no minimizer. Moreover, $f$ may be $\mu$-P{\L} for some $\mu>0$,
even if the individual functions $f_i$ are merely weakly convex.\medskip

	\begin{figure}[t]
	\begin{algorithm}[H] 
		\caption{\algn{PAGE}}\label{alg1}
		\begin{algorithmic}
			\STATE  \textbf{input:} initial estimate $x^0$, initial gradient estimate $g^0$, stepsize $\gamma>0$, probability $p\in (0,1]$.
			\FOR{$t=0, 1, \ldots$}
			\STATE $x^{t+1} \coloneqq x^t - \gamma g^t$
			\STATE flip a coin $\theta^t = (1$ with probability $p$, $0$ otherwise)
			\IF{$\theta^t=0$}
			\STATE pick $i^t \in [n]$ uniformly at random
			\STATE $g^{t+1} \coloneqq g^t +\nabla f_{i^t}(x^{t+1})-\nabla f_{i^t}(x^{t})$
			\ELSE
			\STATE $g^{t+1} \coloneqq \nabla f(x^{t+1})$
			\ENDIF 
			\ENDFOR
		\end{algorithmic}
	\end{algorithm}
	\end{figure}

To solve the problem \eqref{eqp1}, we consider the ProbAbilistic Gradient Estimator algorithm (\algn{PAGE})~\cite{li21},  shown as Algorithm~\ref{alg1}. The method was also proposed independently under the name Loopless-SARAH (\algn{L2S})~\cite{li20a}, as it is a single-loop version of \algn{SARAH} \cite{ngu17}. 
 At each iteration $t\geq 0$, \algn{PAGE} performs a standard GD step
  with small probability $p>0$;  otherwise, it selects an index $i^t$ uniformly at random and evaluates 
$\nabla f_{i^t}$ at two different points to construct a biased gradient estimate for the next update.
 The key advantage of \algn{PAGE} is its drastic reduction in gradient evaluations compared to deterministic \algn{GD},
which iterates $x^{t+1}\coloneqq x^t-\gamma \nabla f(x^t)$. Each GD step requires evaluating all $n$ gradients $\nabla f_i(x^t)$ to compute $\nabla f(x^t)=\frac{1}{n}\sum_{i=1}^n \nabla f_i(x^t)$. By contrast, one iteration of \algn{PAGE} calls on average $pn+2(1-p)$ gradients.  For the typical choice $p=\frac{1}{n}$, this is about 3, making each \algn{PAGE} iteration roughly 
$\frac{n}{3}$ times faster  than a GD step.
\algn{PAGE} reduces to \algn{GD} 
when $p=1$ or $n=1$, with the initialization $g^0 = \nabla f(x^0)$.\medskip

Our contributions are twofold. First, we develop a unified Lyapunov-based analysis of \algn{PAGE} within the $\tau$-weakly convex framework. Second, we establish new sublinear and linear convergence rates, with complexities that improve 
as $\tau$ decreases. 
In particular, we recover the previously known optimal complexities in the general nonconvex regime and our complexities in the convex regime are sharper compared to prior results. 

\subsection{Existing Results}\label{secer}

\algn{PAGE} is highly efficient for nonconvex problems ($\tau=L$). 
Its iteration complexity to reach a 
stationary point $\tilde{x}$ with $\Exp{\sqnorm{\nabla f(\tilde{x})}}\leq \epsilon$, for any $\epsilon>0$, is $\mathcal{O}\big(L\sqrt{n}\epsilon^{-1} + n\big)$, using $\gamma \propto \frac{1}{L\sqrt{n}}$, $p\propto \frac{1}{n}$, $g^0=\nabla f(x^0)$~\cite[Corollary 4]{li20a}\cite[Theorem 1]{li21}. 
The same complexity holds in terms of gradient evaluations, 
and it is optimal \cite{li21}.

Furthermore, if $f$ is $\mu$-P{\L} for some $\mu>0$, 
\algn{PAGE} converges linearly: the complexity to reach a point $\tilde{x}$ with $\Exp{f(\tilde{x})-f^\star}\leq \epsilon$ is $\mathcal{O}\big((\sqrt{n}\kappa + n)\log \epsilon^{-1}\big)$, with same parameter choice $\gamma \propto \frac{1}{L\sqrt{n}}$, $p\propto \frac{1}{n}$, $g^0=\nabla f(x^0)$ \cite[Theorem 5]{li21}. This complexity is also optimal \cite{bai24}.

These results were further refined by Tyurin at al.~\cite{tyu23}, who accounted for different smoothness constants $L_i$ of the functions $f_i$ and exploited their structural homogeneity.\medskip

In the convex setting ($\tau=0$), existing analyses of \algn{PAGE} allow a stepsize  $\gamma$ of order  $\frac{1}{L}$ instead of $\frac{1}{L\sqrt{n}}$, but at the cost of worse 
complexity.  With $\gamma\propto\frac{1}{L}$ and $p=\frac{1}{n}$, 
the complexity  to reach a point $\tilde{x}$ with $\Exp{\sqnorm{\nabla f(\tilde{x})}}\leq \epsilon$ is $\mathcal{O}(Ln\epsilon^{-1} + n)$ \cite[Corollary 2]{li20a}. In the strongly convex setting, with $\gamma < \frac{2}{3L}$ and $p=\frac{1}{\kappa^2}$, 
a modified version of \algn{PAGE} called \algn{L2S-SC} 
has complexity $\mathcal{O}\big((\kappa^2 + n)\log \epsilon^{-1}\big)$  \cite[Corollary 5]{li20a}.    
Since strong convexity implies the P{\L} condition, 
 this can be compared with the complexity $\mathcal{O}\big((\sqrt{n}\kappa + n)\log \epsilon^{-1}\big)$. Depending on the relative values of $n$ and $\kappa$, $\kappa^2$ may be better or worse than $\sqrt{n}\kappa$. 
A notable drawback of \algn{L2S-SC} is that it requires knowledge of $\mu$  in order to set $p$.
Finally, if every component $f_i$ is $\mu$-strongly convex, a much stronger assumption than $f$ being strongly convex, the same authors showed that,  with $\gamma < \frac{2}{3L}$ and $p=\frac{1}{\kappa}$, \algn{L2S-SC} 
achieves $\mathcal{O}\big((\kappa +n)\log \epsilon^{-1}\big)$ \cite[Corollary 6]{li20a}.  Our new analysis improves upon all of these results in the convex setting.

\section{New Convergence Results}

Our analysis of \algn{PAGE} relies on a new Lyapunov function that captures both objective decrease and variance
reduction. Leveraging this formulation, we first establish linear convergence under the P{\L} condition. We then provide sublinear
rates in the general weakly convex setting. Our Lyapunov function $\Psi^t$ is defined for every iteration index $t\geq 0$ as
\begin{equation}
\Psi^t \coloneqq f(x^t)-f^\star -b\gamma \sqnorm{\nabla f(x^t)} + \frac{a \gamma}{p} \sqnorm{g^t-\nabla f(x^t)} + \frac{b \gamma}{p} \sqnorm{g^t},\label{eqly}
\end{equation}
for some constants $a>0$ and $b\in[ 0, \frac{1}{2\gamma L}]$. Since $f$ is $L$-smooth, we have 
$f(x^t)-f^\star \geq  \frac{1}{2L}\sqnorm{\nabla f(x^t)}$
\cite[eq. 3.5]{bub15}, 
so that the condition $b\gamma  \leq \frac{1}{2L}$ guarantees that $\Psi^t$ is a valid Lyapunov function: we have $\Psi^t\geq 0$, and $\Psi^t \rightarrow 0$ implies $\sqnorm{g^t}\rightarrow 0$, $\sqnorm{\nabla f(x^t)}\rightarrow 0$, and
$f(x^t)-f^\star \rightarrow 0$.

The novelty of this Lyapunov function resides in the term $b \gamma\big(\frac{1}{p} \sqnorm{g^t}-\sqnorm{\nabla f(x^t)}\big)$. When expanding its conditional expectation, a negative inner product arises, which can be exploited to cancel positive terms appearing elsewhere, 
via the following lemma.
\begin{lemma}\label{lemm2}
Let $g:\mathcal{X}\rightarrow \mathbb{R}$ be an $L$-smooth and $\tau$-weakly convex function for some $L>0$ and $\tau\in [0,L]$. Then, for every $(x,y)\in\mathcal{X}^2$, 
\begin{align*}
\sqnorm{\nabla g(x)- \nabla g(y)}&\leq (L-\tau)\langle \nabla g(x) - \nabla g(y),x-y\rangle + L\tau \sqnorm{x-y}.
\end{align*}
\end{lemma}
For  $\tau=0$ and $\tau=L$, this inequality reduces to the $\frac{1}{L}$-cocoercivity and $L$-Lipschitzness of the gradient, respectively.\medskip

Our first result establishes the linear convergence of \algn{PAGE} under 
the P{\L} condition.

\begin{theorem}[linear convergence]\label{theo1}
Suppose that $f$ is $\mu$-P{\L} for some $\mu>0$. 
In \algn{PAGE} with Lyapunov function \eqref{eqly}, suppose that
\begin{equation*}
\gamma \leq 
\frac{1}{L\left(1+\sqrt{ \frac{4\tau}{L+\tau}}\sqrt{  \frac{1-p}{p}}\right)}, 
\end{equation*}
with the additional requirement $\gamma<\frac{1}{L}$ if $\tau=0$, 
\begin{equation*}
a= 1-\frac{\gamma (L-\tau)}{2}\in \left( \frac{1}{2},1\right], \quad b= \frac{\gamma (L-\tau)}{2}\in \left[0, \frac{1}{2}\right).
\end{equation*}
Then $b\gamma  < \frac{1}{2L}$, and 
\algn{PAGE}
converges linearly:  for every $t\geq 0$, 
\begin{equation*}
\Exp{\Psi^{t}}\leq \rho^t \Psi^0,
\end{equation*}
where 
\begin{equation*}
\rho\coloneqq \max\left(1-(1-2b)\gamma\mu,1-p\left(1-\frac{1}{2a}\right)\right)<1.
\end{equation*}
Also, $\Psi^t\rightarrow 0$ 
almost surely.
\end{theorem}

\begin{corollary}[complexity]\label{corc1}
Under the assumptions of Theorem~\ref{theo1}, suppose that
\begin{equation*}
\gamma = \frac{\eta}{L\left(1+\sqrt{ \frac{4\tau}{L+\tau}}\sqrt{  \frac{1-p}{p}}\right)},
\end{equation*}
for some constant $\eta\in(0,1)$.
Then the number of iterations to reach $\epsilon$-accuracy, i.e., $\Exp{\Psi^{t}}\leq \epsilon$, is
\begin{equation*}
\mathcal{O}\left(\left(\kappa +\kappa\sqrt{ \frac{\tau}{Lp}} +\frac{1}{p} \right)\log\left( \frac{\Psi^0}{\epsilon}\right)\right).
\end{equation*}
Since each iteration requires $np+2(1-p)$ gradient calls on average, the complexity in terms of gradient evaluations is
\begin{equation*}
\mathcal{O}\left(\left(\kappa +\kappa pn+\kappa\sqrt{ \frac{\tau}{Lp}} +\kappa\sqrt{ \frac{\tau}{L}}n\sqrt{p} 
+\frac{1}{p} +n\right)\log\left( \frac{\Psi^0}{\epsilon}\right)\right).
\end{equation*}
With the  choice  $p \propto \frac{1}{n}$, this simplifies to
\begin{equation*}
\mathcal{O}\left(\left(\kappa +\kappa\sqrt{ \frac{\tau n}{L}} 
 +n\right)\log\left( \frac{\Psi^0}{\epsilon}\right)\right).
\end{equation*}
\end{corollary}
When $\tau=L$, we recover 
the optimal complexity stated in Section~\ref{secer}. For $\tau<L$, our results are new. In particular, in the convex case $\tau=0$, 
with $\gamma\propto \frac{1}{L}$ and $p \propto \frac{1}{n}$, the complexity  of \algn{PAGE}  is 
\begin{equation*}
\mathcal{O}\left((\kappa 
 +n)\log\left( \frac{\Psi^0}{\epsilon}\right)\right).
\end{equation*}
When the components $f_i$ are convex ($\tau=0$) and $f$ is $\mu$-strongly convex for some $\mu>0$,
there exist  algorithms with complexity $\widetilde{\mathcal{O}}\big((\sqrt{\kappa n} +n)\log \epsilon^{-1}\big)$ \cite{all23,kov202}, where the tilde hides logarithmic factors. 
This complexity is known to be optimal \cite{han24}.  The assumption that $f$ is  $\mu$-P{\L} is strictly weaker than $\mu$-strong convexity. In particular, when minimizing a single function (equivalently, Problem \eqref{eqp1} with $n=1$), the $\mathcal{O}(\kappa\log \epsilon^{-1})$ complexity achieved by standard \algn{GD} is optimal and cannot be improved \cite{yue23}. We now establish that the $\widetilde{\mathcal{O}}\big((\kappa +n)\log \epsilon^{-1}\big)$  complexity  achieved by \algn{PAGE} is optimal:

\begin{theorem}[Optimality of \algn{PAGE} in the convex $\mu$-P{\L} regime ]\label{theoo}
Suppose $\tau=0$ and $f$ is $\mu$-P{\L} for some $\mu>0$. 
A lower bound on the number of gradient evaluations is 
 \begin{equation*}
\widetilde{\Omega}\left((\kappa 
 +n)\log(\epsilon^{-1})\right).
\end{equation*}
Therefore, the complexity 
\begin{equation*}
\widetilde{\mathcal{O}}\left((\kappa 
 +n)\log(\epsilon^{-1})\right)
\end{equation*}
of  \algn{PAGE}  established in Corollary \ref{corc1} is optimal.
\end{theorem}

\begin{proof}Assume $\tau=0$ (convex components) and that $f$ is $\mu$-P{\L}. We combine lower bounds from two worst-case scenarios. First, since $\mu$-strong convexity implies the $\mu$-P{\L} condition, the lower bound $\widetilde{\Omega}\big((n + \sqrt{n\kappa})\log \epsilon^{-1} \big)$ \cite{han24} applies to our setting. Second, the lower bound $\widetilde{\Omega}(\kappa\log \epsilon^{-1})$ for minimizing a single convex $\mu$-P{\L} function \cite{yue23} naturally extends to the $n$-component setting by choosing identical functions ($f_1 = \dots = f_n = f$). Taking the maximum of these two bounds yields a global lower bound of
$\widetilde{\Omega}\big((\kappa + n + \sqrt{n\kappa})\log \epsilon^{-1} \big) = \widetilde{\Omega}\big((\kappa + n)\log \epsilon^{-1}\big)$. This matches the complexity of \algn{PAGE} up to logarithmic factors, establishing its optimality. \qed\end{proof}

Our  second main result establishes the sublinear convergence of \algn{PAGE} in the general weakly convex case.

\begin{theorem}[sublinear convergence]\label{theo2}
In \algn{PAGE} with Lyapunov function \eqref{eqly}, suppose that
\begin{equation*}
\gamma \leq 
\frac{1}{L\left(1+\sqrt{ \frac{2\tau}{L+\tau}}\sqrt{  \frac{1-p}{p}}\right)}, \quad
a= \frac{1}{2}, \quad b= \frac{\gamma (L-\tau)}{4-2\gamma(L-\tau)},
\end{equation*}
with the additional requirement $\gamma<\frac{1}{L}$ if $\tau=0$. 
Then \algn{PAGE} converges sublinearly:  for every $T\geq 0$, define $\tilde{x}^T\coloneqq x^t$ for some $t$ chosen uniformly at random in $\{0,\ldots,T\}$. Then 
\begin{equation*}
 \Exp{\sqnorm{\nabla f(\tilde{x}^T)}}\leq \frac{1}{T+1}\frac{2\Psi^0 }{\gamma-\gamma^2(L-\tau)}.
\end{equation*}
 Also, $\sqnorm{\nabla f(x^t)}\rightarrow 0$ 
almost surely.
\end{theorem}

We now analyze the complexity required  to achieve $\Exp{\sqnorm{\nabla f(\tilde{x}^T)}}\leq \epsilon$, for any $\epsilon>0$. 
It is standard practice to express the complexity in terms of the initial gap
$\Delta_0\coloneqq f(x^0)-f^\star$, rather than
\begin{equation*}
\Psi^0 = \Delta_0-b\gamma \sqnorm{\nabla f(x^0)} + \frac{a \gamma}{p} \sqnorm{g^0-\nabla f(x^0)} + \frac{b \gamma}{p} \sqnorm{g^0}.
\end{equation*}
Thus, it remains to choose $g^0$  appropriately, using the bound
$\sqnorm{\nabla f(x^0)}\leq 2L \Delta_0$. 
We first  consider the case of a large stepsize $\gamma$ and arbitrary initialization $g^0$.

\begin{corollary}[complexity]\label{corc3}
Under the assumptions of Theorem~\ref{theo2}, suppose that
\begin{equation*}
\gamma = \frac{\eta}{L\left(1+\sqrt{ \frac{2\tau}{L+\tau}}\sqrt{  \frac{1-p}{p}}\right)}
\end{equation*}
for some constant $\eta\in(0,1)$, and
$\sqnorm{g^0}=\mathcal{O}(\sqnorm{\nabla f(x^0)})$, for example, $g^0=0$. Then 
\begin{equation*}
 \Exp{\sqnorm{\nabla f(\tilde{x}^T)}}=\mathcal{O}\left( \frac{\Psi^0 }{T\gamma}\right)=\mathcal{O}\left(\frac{\Delta_0 L }{Tp}\right).
\end{equation*}
Thus, if $p \geq \frac{1}{n}$,  the complexity of \algn{PAGE} is 
$\mathcal{O}(\Delta_0 Ln \epsilon^{-1})$ gradient evaluations.
\end{corollary}

This result is new. In the convex case $\tau=0$, the complexity $\mathcal{O}(\Delta_0 Ln \epsilon^{-1})$ improves upon the previously known $\mathcal{O}(\Delta_0 Ln\epsilon^{-1} + n)$ \cite[Corollary 2]{li20a}, thanks to the flexibility of initializing $g^0$  arbitrarily.

To obtain a better complexity, note that in $\Psi^0$, $b$ is proportional to $\gamma$, and can therefore be made small by choosing 
$\gamma$ small. In contrast,  $a=\frac{1}{2}$ is fixed. Hence, the term   $\frac{a \gamma}{p} \sqnorm{g^0-\nabla f(x^0)}$  must be eliminated by setting  $g^0=\nabla f(x^0)$, which generally requires  $n$ gradient evaluations. 

\begin{corollary}[complexity with $g^0=\nabla f(x^0)$]\label{corc4}
Under the assumptions of Theorem~\ref{theo2}, suppose that
$g^0=\nabla f(x^0)$. Then 
\begin{equation*}
 \Exp{\sqnorm{\nabla f(\tilde{x}^T)}}=\mathcal{O}\left( \frac{\Psi^0 }{T\gamma}\right)=\mathcal{O}\left(\frac{\Delta_0 L }{T}
\left (\frac{1}{\gamma L}+\frac{\gamma (L-\tau)}{p}\right)
 \right).
 \end{equation*}
 In particular, with $p \geq \frac{1}{n}$ and $\gamma \propto\frac{1}{\sqrt{n}L}$, regardless of $\tau$, the iteration complexity and gradient evaluation complexity of \algn{PAGE} are
$\mathcal{O}(\Delta_0 L \sqrt{n} \epsilon^{-1})$ and $\mathcal{O}(\Delta_0 L \sqrt{n} \epsilon^{-1}+n)$, respectively.
\end{corollary}
This complexity, which is optimal in the general nonconvex setting (see Section~\ref{secer}), 
does not improve as $\tau$ decreases. This is because we express the worst-case complexity with respect to $\Delta^0$, not $\Psi^0$, and taking a larger stepsize $\gamma$ worsens the dependence of $\Psi^0$ on $\Delta^0$. There exist methods whose complexity improves as $\tau$ decreases, see Han et al.~\cite{han24} for a recent review.

\section{Conclusion}

We developed a unified Lyapunov analysis of the \algn{PAGE} stochastic algorithm in the broad framework of minimizing a sum of smooth, weakly convex functions. Our results establish both linear convergence under the Polyak--{\L}ojasiewicz condition and sublinear convergence in the general regime, yielding complexities that interpolate smoothly between the nonconvex and convex cases.
Specifically, our analysis yields the following two main results: First, when $f$ is $\mu$-P{\L}, the gradient evaluation complexity of \algn{PAGE} is
\begin{equation*}\widetilde{\mathcal{O}}\left(\left(\kappa +\kappa\sqrt{ \frac{\tau n}{L}} +n\right)\log (\epsilon^{-1})\right),\end{equation*}
which is optimal at both extreme cases $\tau=0$ and $\tau=L$.
Second, in the general weakly convex regime, the complexity is
\begin{equation*}\mathcal{O}\left((f(x^0)-f^\star) L \sqrt{n} \epsilon^{-1}+n\right),\end{equation*}
which is optimal in the nonconvex case ($\tau=L$).
Thus, our derived rates improve upon prior results in the convex setting while recovering known optimal bounds in the nonconvex setting. Ultimately, our study highlights the versatility and efficiency of \algn{PAGE}, without requiring restrictive assumptions.

\begin{acknowledgements}
This work was supported by funding from King Abdullah University of Science and Technology (KAUST): 

\noindent i) KAUST Baseline Research Scheme, 

\noindent ii) Center of Excellence for Generative AI (award no.\ 5940), 

\noindent iii) Competitive Research Grant (CRG) Program (award no.\ 6460), 

\noindent iv) SDAIA-KAUST Center of Excellence in Data Science and Artificial Intelligence (SDAIA-KAUST AI).
\end{acknowledgements}

{\small

\noindent\textbf{Data availability statement\ \ }No dataset has been used to support the findings of this paper, which is of theoretical nature.\medskip

\noindent\textbf{Competing Interests\ \ }The authors have no competing interests to declare that are relevant to the content of this article.
}

\appendix

\section{Proofs}

\begin{lemma}\label{lemm1}
Let $g:\mathcal{X}\rightarrow \mathbb{R}$ be an $L$-smooth function for some $L>0$. Then $g$ is $L$-weakly convex. 
\end{lemma}
\begin{proof}
Let $(x,y)\in\mathcal{X}^2$. We have 
\begin{align*}
\sqnorm{\nabla g(x)+Lx - \nabla g(y)-Ly}&=\sqnorm{\nabla g(x) - \nabla g(y)}+\sqnorm{Lx-Ly}\\
&\quad+2L\langle \nabla g(x) - \nabla g(y),x-y\rangle\\
&\leq L^2 \sqnorm{x-y}+L^2 \sqnorm{x-y}+2L\langle \nabla g(x) - \nabla g(y),x-y\rangle \\
&= 2L\langle \nabla g(x) +Lx- \nabla g(y)-Ly,x-y\rangle.
\end{align*}
Thus,  $\nabla g + L\mathrm{Id}$ is  monotone (and even $\frac{1}{2L}$-cocoercive),
 where $\mathrm{Id}$ denotes the identity. Since $\nabla g + L\mathrm{Id}$ is the gradient of $g+\frac{L}{2}\|\cdot\|^2$, and a differentiable function is convex if and only if its gradient is monotone \cite[Proposition 17.7]{bau17}, $g+\frac{L}{2}\|\cdot\|^2$ is convex.
 \end{proof}

\subsection{Proof of Lemma~\ref{lemm2}}

Let $(x,y)\in\mathcal{X}^2$. Since $g+\frac{\tau}{2}\|\cdot\|^2$ is convex and $(L+\tau)$-smooth, its gradient is $\frac{1}{L+\tau}$-cocoercive according to the Baillon--Haddad theorem \cite[Corollary 18.17]{bau17}, so that
\begin{align*}
\sqnorm{\nabla g(x)+\tau x - \nabla g(y)-\tau y} &\leq (L+\tau)\langle \nabla g(x) +\tau x- \nabla g(y)-\tau y,x-y\rangle\\
&=(L+\tau)\langle \nabla g(x) - \nabla g(y),x-y\rangle +  (L+\tau)\tau \sqnorm{x-y}.
\end{align*}
Hence,
\begin{align*}
\sqnorm{\nabla g(x)- \nabla g(y)}&=\sqnorm{\nabla g(x)+\tau x- \nabla g(y)-\tau y}\!-\!2\tau
\langle \nabla g(x) - \nabla g(y),x-y\rangle 
\!-\!\tau^2\sqnorm{x-y}\\
&\leq (L-\tau)\langle \nabla g(x) - \nabla g(y),x-y\rangle + L\tau \sqnorm{x-y}.
\end{align*}\qed

\subsection{Proof of Theorem~\ref{theo1}}\label{secproof1}

For every $t\geq 0$, let $\mathcal{F}^t$ denote the $\sigma$-algebra generated by the random variables $x^0,g^0,\ldots, x^t,g^t$. For the moment, assume that $a$ and $b$ in \eqref{eqly} satisfy  $a>0$, 
$0\leq b\leq \frac{1}{2}$, $b\gamma  \leq \frac{1}{2L}$, without additional restrictions.

Let $t\geq 0$. We now derive upper bounds for the terms in $\Exp{\Psi^{t+1}\;|\;\mathcal{F}^t}$.

\noindent \textbf{First}, we have the descent property \cite[Lemma 4]{ric21}
\begin{align*}
f(x^{t+1}) -f^\star &\leq f(x^t)  -f^\star -\frac{\gamma}{2} \sqnorm{\nabla f(x^t)} +\frac{ \gamma }{2}\sqnorm{g^{t}-\nabla f(x^t)}+ \left(\frac{ L}{2}-\frac{1}{2\gamma}\right)\sqnorm{x^{t+1}-x^t}\\
&=f(x^t)  -f^\star -\frac{\gamma}{2} \sqnorm{\nabla f(x^t)} +\frac{ \gamma }{2}\sqnorm{g^{t}-\nabla f(x^t)}+ \left(\frac{ L\gamma^2}{2}-\frac{\gamma}{2}\right)\sqnorm{g^t}.
\end{align*}
\textbf{Second}, we have 
\begin{align*}
\Exp{\sqnorm{g^{t+1}-\nabla f(x^{t+1})}\;|\;\mathcal{F}^t} &= (1-p)\Exp{\sqnorm{g^{t+1}-\nabla f(x^{t+1})}\;|\;\theta^t=0,\mathcal{F}^t},
\end{align*}
and by decomposing 
the expected squared norm into its squared mean and variance, we obtain
\begin{align*}
&\Exp{\sqnorm{g^{t+1}-\nabla f(x^{t+1})}\;|\;\mathcal{F}^t} = (1-p)\sqnorm{g^{t}-\nabla f(x^t)}\\
&\qquad +\frac{1-p}{n}\sum_{i=1}^n\sqnorm{\nabla f_i(x^{t+1})-\nabla f_i(x^{t})-\nabla f(x^{t+1})+\nabla f(x^t)}.
\end{align*}
\textbf{Third}, we have
\begin{align*}
\Exp{\sqnorm{g^{t+1}}\;|\;\mathcal{F}^t} &= p \sqnorm{\nabla f(x^{t+1})}+ (1-p)\Exp{\sqnorm{g^{t+1}}\;|\;\theta^t=0,\mathcal{F}^t}\\
&=p \sqnorm{\nabla f(x^{t+1})}+ (1-p)\sqnorm{g^{t}+\nabla f(x^{t+1})-\nabla f(x^t)}\\
&\quad +\frac{1-p}{n}\sum_{i=1}^n\sqnorm{\nabla f_i(x^{t+1})-\nabla f_i(x^{t})-\nabla f(x^{t+1})+\nabla f(x^t)}.
\end{align*}
Using the fact that $g^t=-\frac{1}{\gamma}(x^{t+1}-x^t)$, we obtain
\begin{align*}
\sqnorm{g^{t}+\nabla f(x^{t+1})-\nabla f(x^t)} &= \sqnorm{g^{t}} + \sqnorm{\nabla f(x^{t+1})-\nabla f(x^t)} +2\!\left\langle g^{t},\nabla f(x^{t+1})-\nabla f(x^t)\right\rangle\\
&=\sqnorm{g^{t}} + \sqnorm{\nabla f(x^{t+1})-\nabla f(x^t)}\\
&\quad -\frac{2}{\gamma}\left\langle x^{t+1}-x^t,\nabla f(x^{t+1})-\nabla f(x^t)\right\rangle.
\end{align*}
The negative inner product $-\frac{2}{\gamma}\left\langle x^{t+1}-x^t,\nabla f(x^{t+1})-\nabla f(x^t)\right\rangle$ that appears is crucial, as it enables the cancellation of positive terms using Lemma~\ref{lemm2}.

Since $f=\frac{1}{n}\sum_{i=1}^n f_i$, we have
\begin{align*}
&\frac{1}{n}\sum_{i=1}^n\sqnorm{\nabla f_i(x^{t+1})-\nabla f_i(x^{t})-\nabla f(x^{t+1})+\nabla f(x^t)}\\
&\quad= \frac{1}{n}\sum_{i=1}^n\sqnorm{\nabla f_i(x^{t+1})-\nabla f_i(x^{t})} -\sqnorm{\nabla f(x^{t+1})-\nabla f(x^t)}.
\end{align*}
Therefore,
\begin{align*}
\Exp{\sqnorm{g^{t+1}}\;|\;\mathcal{F}^t} &=p \sqnorm{\nabla f(x^{t+1})}+ (1-p)\sqnorm{g^{t}}\\
&\quad -\frac{2(1-p)}{\gamma}\left\langle x^{t+1}-x^t,\nabla f(x^{t+1})-\nabla f(x^t)\right\rangle\\
 &\quad +\frac{1-p}{n}\sum_{i=1}^n\sqnorm{\nabla f_i(x^{t+1})-\nabla f_i(x^{t})}.
\end{align*}

Moreover, by Lemma~\ref{lemm2} and the smoothness and weak convexity of the functions $f_i$, we obtain
\begin{align*}
\frac{1}{n}\sum_{i=1}^n \sqnorm{\nabla f_i(x^{t+1})- \nabla f_i(x^t)}&\leq\frac{1}{n}\sum_{i=1}^n \Big( (L-\tau)\big\langle \nabla f_i(x^{t+1}) - \nabla f_i(x^t),x^{t+1}-x^t\big\rangle \\
&\quad\left.{}+ L\tau \sqnorm{x^{t+1}-x^t}\right)\\
&= (L-\tau)\big\langle \nabla f(x^{t+1}) - \nabla f(x^t),x^{t+1}-x^t\big\rangle \\
&\quad+ L\tau \sqnorm{x^{t+1}-x^t}.
\end{align*}
The inner products $(L-\tau)\big\langle \nabla f(x^{t+1}) - \nabla f(x^t),x^{t+1}-x^t\big\rangle$ will be canceled using the negative inner product $-\frac{2}{\gamma}\left\langle x^{t+1}-x^t,\nabla f(x^{t+1})-\nabla f(x^t)\right\rangle$ that appeared above, while the terms 
$\sqnorm{x^{t+1}-x^t}$ will be canceled by choosing $\gamma$ sufficiently small and using the negative term $(\frac{ L}{2}-\frac{1}{2\gamma})\sqnorm{x^{t+1}-x^t}$ arising in the descent property.

Thus, 
\begin{align*}
\Exp{\sqnorm{g^{t+1}-\nabla f(x^{t+1})}\;|\;\mathcal{F}^t} 
&\leq (1-p)\sqnorm{g^{t}-\nabla f(x^t)}\!+\!
\frac{1-p}{n}\sum_{i=1}^n \sqnorm{\nabla f_i(x^{t+1})- \nabla f_i(x^t)}\\
&\leq (1-p)\sqnorm{g^{t}-\nabla f(x^t)}+ (1-p)L\tau \sqnorm{x^{t+1}-x^t}\\
&\quad +(1-p)(L-\tau)  \left\langle \nabla f(x^{t+1}) - \nabla f(x^t),x^{t+1}-x^t\right\rangle
\end{align*}
and
\begin{align*}
\Exp{\sqnorm{g^{t+1}}\;|\;\mathcal{F}^t} 
&\leq p \sqnorm{\nabla f(x^{t+1})}+ (1-p)\sqnorm{g^{t}}\\
&\quad -\frac{2(1-p)}{\gamma}\left\langle \nabla f(x^{t+1})-\nabla f(x^t),x^{t+1}-x^t\right\rangle\\
 &\quad +(1-p)(L-\tau)\left\langle \nabla f(x^{t+1}) - \nabla f(x^t),x^{t+1}-x^t\right\rangle\\
&\quad +(1-p)L\tau \sqnorm{x^{t+1}-x^t}\\
 &= p\sqnorm{\nabla f(x^{t+1})}+ (1-p)\big(1+L\tau\gamma^2\big)
\sqnorm{g^{t}} \\
&\quad +(1-p)\left(L- \tau-\frac{2}{\gamma}\right) \left\langle \nabla f(x^{t+1}) - \nabla f(x^t),x^{t+1}-x^t\right\rangle .
\end{align*}
Hence, combining the inequalities for these terms, we obtain
\begin{align*}
\Exp{\Psi^{t+1}\;|\;\mathcal{F}^t}&\leq
f(x^t)  -f^\star -\frac{\gamma}{2} \sqnorm{\nabla f(x^t)} +\frac{ \gamma }{2}\sqnorm{g^{t}-\nabla f(x^t)}+ \left(\frac{ L\gamma^2}{2}-\frac{\gamma}{2}\right)\sqnorm{g^t} \\
&\quad -b\gamma\sqnorm{\nabla f(x^{t+1})}
+\frac{a\gamma}{p} (1-p)\sqnorm{g^{t}-\nabla f(x^t)}
+ a\gamma^3 L\tau \frac{1-p}{p}\sqnorm{g^t}\\
&\quad+a\gamma (L-\tau) \frac{1-p}{p}\left\langle \nabla f(x^{t+1}) - \nabla f(x^t),x^{t+1}-x^t\right\rangle+b\gamma\sqnorm{\nabla f(x^{t+1})}\notag\\
&\quad+\frac{b\gamma}{p}  (1-p)\big(1+L\tau\gamma^2\big)\sqnorm{g^{t}}\\
&\quad+ \big(b\gamma (L-\tau)-2b\big) \frac{1-p}{p}\left\langle \nabla f(x^{t+1}) - \nabla f(x^t),x^{t+1}-x^t\right\rangle\notag\\
&= f(x^t)  -f^\star  -\frac{\gamma}{2} \sqnorm{\nabla f(x^t)}+\frac{a\gamma}{p}\left(1-p\left(1-\frac{1 }{2a}\right)  \right)\sqnorm{g^{t}-\nabla f(x^t)}\notag\\
&\quad+\frac{b\gamma}{p}  (1-p)\sqnorm{g^{t}}
+\frac{\gamma}{2}\left(L\gamma-1+ 2(a+b)\gamma^2 L\tau \frac{1-p}{p}\right)\sqnorm{g^{t}}\notag\\
&\quad+ \big((a+b)\gamma (L-\tau)-2b\big) \frac{1-p}{p}\left\langle \nabla f(x^{t+1}) - \nabla f(x^t),x^{t+1}-x^t\right\rangle.
\end{align*}
To cancel the inner product, we set
\begin{equation}
b\coloneqq \frac{a\gamma (L-\tau)}{2-\gamma(L-\tau)},\label{eqbb}
\end{equation}
so that $(a+b)\gamma (L-\tau)-2b=0$. Note that $\gamma\leq \frac{1}{L}$ implies $2-\gamma(L-\tau)\geq 1$. Moreover,
a sufficient condition for $L\gamma-1+ 2(a+b)\gamma^2 L\tau \frac{1-p}{p}\leq 0$ is  \cite[Lemma 5]{ric21}
\begin{equation}
\gamma \leq \frac{1}{L+\sqrt{2(a+b)L\tau \frac{1-p}{p}}}.\label{eqlllo}
\end{equation}
To avoid the dependence of $b$ on $\gamma$, we derive from $\gamma\leq \frac{1}{L}$ that 
\begin{equation*}
b\leq  \frac{a (L-\tau)}{L+\tau},
\end{equation*}
so that $a+b\leq  \frac{2aL}{L+\tau}$. Therefore, a sufficient condition for \eqref{eqlllo} is
\begin{equation}
\gamma \leq \frac{1}{L+\sqrt{ \frac{4aL}{L+\tau}L\tau \frac{1-p}{p}}}=\frac{1}{L\left(1+\sqrt{ \frac{4a\tau}{L+\tau}}\sqrt{  \frac{1-p}{p}}\right)}.\label{eqlllo2}
\end{equation}
If we additionally assume $a\leq 1$,  the dependence on $a$ can be eliminated with the sufficient condition 
\begin{equation*}
\gamma \leq \frac{1}{L\left(1+\sqrt{ \frac{4\tau}{L+\tau}}\sqrt{  \frac{1-p}{p}}\right)},
\end{equation*}
as stated in the theorem.

Therefore, since the sufficient conditions above are supposed to hold, we have
\begin{align*}
\Exp{\Psi^{t+1}\;|\;\mathcal{F}^t}&\leq
 f(x^t)  -f^\star  -\frac{\gamma}{2} \sqnorm{\nabla f(x^t)}+\frac{a\gamma}{p}\left(1-p\left(1-\frac{1 }{2a}\right)  \right)\sqnorm{g^{t}-\nabla f(x^t)}\\
&\quad+\frac{b\gamma}{p}  (1-p)\sqnorm{g^{t}}\\
&= f(x^t)  -f^\star  -\gamma\left(\frac{1}{2}-b\right)\sqnorm{\nabla f(x^t)} -b\gamma \sqnorm{\nabla f(x^t)}\\
&\quad+\frac{a\gamma}{p}\left(1-p\left(1-\frac{1}{2a}\right)  \right)\sqnorm{g^{t}-\nabla f(x^t)}+\frac{b\gamma}{p}  (1-p)\sqnorm{g^{t}}.
\end{align*}
To obtain a contraction, we need $a>\frac{1}{2}$ and $b<\frac{1}{2}$. The choice 
\begin{equation*}
a= 1-\frac{\gamma (L-\tau)}{2}, \quad b= \frac{\gamma (L-\tau)}{2},
\end{equation*}
with the additional condition $\gamma<\frac{1}{L}$ if $\tau=0$, satisfies  these two inequalities, as well as \eqref{eqbb}, $a\leq 1$, and $b \gamma \leq \frac{1}{2L}$.

Finally, under the assumption that $f$ is $\mu$-P{\L}, we have
\begin{equation*}
\|\nabla f(x^t)\|^2\geq 2\mu\big(f(x^t)-f^\star\big).
\end{equation*}
Therefore,
\begin{align}
\Exp{\Psi^{t+1}\;|\;\mathcal{F}^t}&\leq  \big(1-(1-2b)\gamma\mu\big)\big(f(x^t)  -f^\star\big)  -b\gamma \sqnorm{\nabla f(x^t)}\notag\\
&\quad+\frac{a\gamma}{p}\left(1-p\left(1-\frac{1}{2a}\right)  \right)\sqnorm{g^{t}-\nabla f(x^t)}+\frac{b\gamma}{p}  (1-p)\sqnorm{g^{t}}\notag\\
&\leq  \big(1-(1-2b)\gamma\mu\big)\big(f(x^t)  -f^\star  -b\gamma \sqnorm{\nabla f(x^t)}\big)\notag\\
&\quad+\frac{a\gamma}{p}\left(1-p\left(1-\frac{1}{2a}\right)  \right)\sqnorm{g^{t}-\nabla f(x^t)}+\frac{b\gamma}{p}  (1-p)\sqnorm{g^{t}}\notag\\
&\leq \max\left(1-(1-2b)\gamma\mu,1-p\left(1-\frac{1}{2a}\right),1-p\right)\Psi^t\notag\\
&= \max\left(1-(1-2b)\gamma\mu,1-p\left(1-\frac{1}{2a}\right)\right)\Psi^t.\label{eqrec2b}
\end{align}
Using the tower rule of expectations, we can unroll the recursion in \eqref{eqrec2b} to obtain the unconditional expectation of $\Psi^{t+1}$. Moreover, using classical results on supermartingale convergence \cite[Proposition A.4.5]{ber15}, it follows from \eqref{eqrec2b} that $\Psi^t\rightarrow 0$ almost surely. \qed

\subsection{Proof of Theorem~\ref{theo2}}\label{secproof2}

We proceed as in Section~\ref{secproof1}, now with $a=\frac{1}{2}$ and, by \eqref{eqbb},
\begin{equation*}
b= \frac{\gamma (L-\tau)}{4-2\gamma(L-\tau)}< \frac{1}{2}.
\end{equation*}
 Resuming from the condition \eqref{eqlllo2}, which is supposed to hold, we have
 \begin{align}
\Exp{\Psi^{t+1}\;|\;\mathcal{F}^t}&\leq
 f(x^t)  -f^\star  -\frac{\gamma}{2} \sqnorm{\nabla f(x^t)}+\frac{a\gamma}{p}\sqnorm{g^{t}-\nabla f(x^t)}+\frac{b\gamma}{p}  (1-p)\sqnorm{g^{t}}\notag\\
&= f(x^t)  -f^\star  -\gamma\left(\frac{1}{2}-b\right)\sqnorm{\nabla f(x^t)} -b\gamma \sqnorm{\nabla f(x^t)}\notag\\
&\quad+\frac{a\gamma}{p}\sqnorm{g^{t}-\nabla f(x^t)}+\frac{b\gamma}{p}  (1-p)\sqnorm{g^{t}}\notag\\
&\leq \Psi^t -\gamma\left(\frac{1}{2}-b\right)\sqnorm{\nabla f(x^t)}.  \label{eqcva}
\end{align}
By telescoping the sum and applying the tower rule of expectations, we establish the summability of the squared gradient norms: for every $T\geq 0$,
\begin{equation*}
\gamma\left(\frac{1}{2}-b\right) \sum_{t=0}^T \Exp{\sqnorm{\nabla f(x^t)}}\leq \Psi^0 - \Exp{\Psi^{T+1}} \leq \Psi^0.
\end{equation*}
This implies that $\Exp{\sqnorm{\nabla f(x^t)}}\rightarrow 0$ as $t\rightarrow +\infty$. Also, \eqref{eqcva} implies that $\sqnorm{\nabla f(x^t)}\rightarrow 0$ almost surely.

Now,  for every $T\geq 0$, we define $\tilde{x}^T\coloneqq x^t$ for $t$ chosen uniformly at random in $\{0,\ldots,T\}$. Then,
\begin{align*}
 \Exp{\sqnorm{\nabla f(\tilde{x}^T)}}= \frac{1}{T+1} \sum_{t=0}^T \Exp{\sqnorm{\nabla f(x^t)}}&\leq \frac{1}{T+1}\frac{2\Psi^0}{\gamma (1-2b)}\leq\frac{1}{T+1}\frac{2\Psi^0 }{\gamma-\gamma^2(L-\tau)}.
\end{align*}\qed

\end{document}